\documentclass[11pt]{amsart}
\usepackage{amsmath}
\usepackage{amsfonts}
\usepackage{amssymb}
\usepackage{amsthm}
\usepackage[usenames]{xcolor}
\usepackage{graphicx}
\usepackage[all]{xy}
\usepackage{url}
\usepackage[abs]{overpic}
\usepackage[hyperfootnotes=false]{hyperref}
\usepackage{verbatim}
\usepackage{todonotes}
\usepackage{tikz}
\usepackage{tikz-cd}
\usetikzlibrary{arrows,calc,matrix,patterns,shadings,backgrounds,fit}
\usepackage{lipsum}

\usepackage[bottom]{footmisc}
\usepackage{environ}

\makeatletter
\newdimen\imgresize@width
\newdimen\imgresize@threshold
\NewEnviron{imgresize}[1]{%
  \pgfmathsetlength\imgresize@width{#1}%
  \edef\BODY{
    \ignorespaces\unexpanded\expandafter{\BODY}\unskip
  }%
  \typeout{}%
  \def\imgresize@scale{1}%
  \tikzset{every picture/.style={scale=\imgresize@scale}}%
  \count@=\@ne
  \typeout{* imgresize: try=\the\count@, scale=\imgresize@scale}%
  \sbox0{\BODY}%
  \pgfmathsetlengthmacro\imgresize@diff{abs(\imgresize@width-\wd0)}%
  \@whiledim\imgresize@diff>\imgresize@threshold\do{%
    \let\imgresize@oldscale\imgresize@scale
    \pgfmathsetmacro\imgresize@scale{%
      \imgresize@width*\imgresize@scale/\the\wd0
    }%
    \ifx\imgresize@scale\imgresize@oldscale
      \PackageWarning{imgresize}{%
        Scale factor does not change anymore,\MessageBreak
        width difference is \imgresize@diff,\MessageBreak
        larger than threshold \the\imgresize@threshold
      }%
      \def\imgresize@diff{0pt}
    \else
      \advance\count@\@ne
      \typeout{* imgresize: try=\the\count@, scale=\imgresize@scale}%
      \sbox0{\BODY}%
      \pgfmathsetlengthmacro\imgresize@diff{abs(\imgresize@width-\wd0)}%
    \fi
  }%
  \usebox{0}%
}
\makeatother

\theoremstyle{definition}
\newtheorem{defn}{Definition}[section]

\newtheorem{rmk}[defn]{Remark}
\newtheorem{question}[defn]{Question}
\theoremstyle{plain}
\newtheorem{thm}[defn]{Theorem}
\newtheorem{lem}[defn]{Lemma}
\newtheorem{prop}[defn]{Proposition}
\newtheorem{cor}[defn]{Corollary}

\def\C{\ensuremath{\mathbb{C}}}

\def\L{\ensuremath{\mathbb{L}}}
\def\N{\ensuremath{\mathbb{N}}}
\def\P{\ensuremath{\mathbb{P}}}

\def\R{\ensuremath{\mathbb{R}}}
\def\Z{\ensuremath{\mathbb{Z}}}

\def\AA{\ensuremath{\mathcal A}}
\def\BB{\ensuremath{\mathcal B}}

\def\EE{\ensuremath{\mathcal E}}
\def\FF{\ensuremath{\mathcal F}}
\def\GG{\ensuremath{\mathcal G}}
\def\HH{\ensuremath{\mathcal H}}
\def\II{\ensuremath{\mathcal I}}

\def\KK{\ensuremath{\mathcal K}}

\def\MM{\ensuremath{\mathcal M}}

\def\OO{\ensuremath{\mathcal O}}
\def\PP{\ensuremath{\mathcal P}}

\def\TT{\ensuremath{\mathcal T}}

\def\ch{\mathop{\mathrm{ch}}\nolimits}

\def\Coh{\mathop{\mathrm{Coh}}\nolimits}

\def\Db{\mathop{\mathrm{D}^{\mathrm{b}}}\nolimits}

\def\Ext{\mathop{\mathrm{Ext}}\nolimits}

\def\hom{\mathop{\mathrm{hom}}\nolimits}
\def\Hom{\mathop{\mathrm{Hom}}\nolimits}

\def\Ker{\mathop{\mathrm{Ker}}\nolimits}

\def\rk{\mathop{\mathrm{rk}}}

\def\Stab{\mathop{\mathrm{Stab}}\nolimits}

\def\Ku{\mathop{\mathrm{Ku}}\nolimits}
\def\Forg{\mathop{\mathrm{Forg}}\nolimits}
\def\dlt{\Delta_{\mathcal B_0}}

\def\chbo{\ch^{-1}_{\BB_0,1}}
\def\chbt{\ch^{-1}_{\BB_0,2}}
\def\chbl{\ch^{-1}_{\BB_0,\leq 2}}
\def\zab{Z_{\alpha,\beta}}
\def\sab{\sigma_{\alpha,\beta}}
\def\mab{\mu_{\alpha,\beta}}

\def\onto{\ensuremath{\twoheadrightarrow}}

\newcommand\blfootnote[1]{%
  \begingroup
  \renewcommand\thefootnote{}\footnote{#1}%
  \addtocounter{footnote}{-1}%
  \endgroup
}

\author{Chunyi Li}
\address{C.\ L.: Mathematics Institute (WMI), University of Warwick, Coventry, CV4 7AL, United Kingdom.}
\email{C.Li.25@warwick.ac.uk}
\urladdr{https://sites.google.com/site/chunyili0401/}
\author{Laura Pertusi}
\address{L.\ P.: Dipartimento di Matematica F.\ Enriques, Universit\`a degli studi di Milano, Via Cesare Saldini 50, 20133 Milano, Italy.}
\email{laura.pertusi@unimi.it}
\urladdr{http://www.mat.unimi.it/users/pertusi/}
\author{Xiaolei Zhao}
\address{X.\ Z.: Department of Mathematics, University of California, Santa Barbara, South Hall 6705, Santa Barbara, CA 93106, USA.}
\email{xlzhao@ucsb.edu}
\urladdr{https://sites.google.com/site/xiaoleizhaoswebsite/}

\begin{document}
\title{Twisted Cubics on Cubic Fourfolds and Stability Conditions}
\maketitle 

\begin{abstract}
We give an interpretation of the Fano variety of lines on a cubic fourfold and of the hyperk\"ahler eightfold, constructed by Lehn, Lehn, Sorger and van Straten from twisted cubic curves in a cubic fourfold non containing a plane, as moduli spaces of Bridgeland stable objects in the Kuznetsov component. As a consequence, we reprove the categorical version of Torelli Theorem for cubic fourfolds, we obtain the identification of the period point of LLSvS eightfold with that of the Fano variety, and we discuss derived Torelli Theorem for cubic fourfolds.  
\end{abstract}

\blfootnote{\textup{2010} \textit{Mathematics Subject Classification}: \textup{14D20}, \textup{14F05}, \textup{14J35}, \textup{14J45}, \textup{18E30}.}

\section{Introduction}

Hyperk\"ahler geometry is a central research area in differential geometry and algebraic geometry. Although much effort has been made, it is still difficult to construct compact hyperk\"ahler varieties. The first known examples are Hilbert schemes of points on K3 surfaces (see \cite{Beauville:HK}), or more generally, moduli spaces of stable sheaves on K3 surfaces (see \cite{Mukai:BundlesK3}). Note that this construction only provides codimension one loci in the polarized moduli spaces.

Another way to construct compact hyperk\"ahler manifolds is via classical algebraic geometry. Let $Y$ be a cubic fourfold and consider the Fano variety $F_Y$ of lines on $Y$. It was shown in \cite{Beauville:HK} that $F_Y$ is a smooth projective hyperk\"ahler fourfold, deformation equivalent to the Hilbert square of a K3 surface. More recently, in \cite{LLSvS} the authors constructed a hyperk\"ahler eightfold $M_Y$ from the irreducible component of the Hilbert scheme of twisted cubic curves on $Y$. One advantage of this approach is that it provides locally complete families of (polarized) projective hyperk\"ahler manifolds.

On the other hand, the geometry of cubic fourfolds has a deep connection with K3 surfaces. The Hodge theoretic interaction was fully explored in the literature, e.g.\ in \cite{Hassett-specialcubic4fold}. From a categorical viewpoint, in \cite{Kuz:fourfold} it is proved that the bounded derived category of coherent sheaves on a cubic fourfold $Y$ admits a semiorthogonal decomposition of the form
\[\Db(Y)=\langle\Ku(Y),\OO_Y,\OO_Y(H),\OO_Y(2H)\rangle.\]
In particular, the \emph{Kuznetsov component} $\Ku(Y)$ is a K3 category, i.e.\ its Serre functor is equal to the homological shift $[2]$. In the celebrated work \cite{BLMS:kuzcomponent}, the authors provide a construction of Bridgeland stability conditions on $\Ku(Y)$ (see Section \ref{section:constructstabonku} for a summary of this construction). In the following, we will denote these stability conditions by $\sigma$. As a consequence, it is possible to study moduli spaces of stable objects in the Kuznetsov component. 

The aim of this paper is to give a description of $F_Y$ and $M_Y$ in terms of moduli spaces of stable objects in the Kuznetsov component, with respect to the Bridgeland stability conditions constructed in \cite{BLMS:kuzcomponent}. 

Recall that the algebraic Mukai lattice of $\Ku(Y)$ always contains an $A_2$ lattice spanned by two classes $\lambda_1$ and $\lambda_2$ (see Section \ref{section:constructstabonku}). We denote by $M_\sigma(v)$ the moduli space of $\sigma$-stable objects in $\Ku(Y)$ with Mukai vector $v$. To each line $\ell$ on $Y$, we can associate an object $P_\ell\in \Ku(Y)$, of Mukai vector $\lambda_1+\lambda_2$ (see Section \ref{section:lines}). Our first result gives a reconstruction of $F_Y$ as follows.

\begin{thm}
\label{thm:Fano_main}
For any line $\ell$ in a cubic fourfold $Y$, the object $P_\ell$ is $\sigma$-stable and the moduli space $M_\sigma(\lambda_1+\lambda_2)$ is isomorphic to the Fano variety $F_Y$.
\end{thm}

As $\Ku(Y)$ is a K3 category, the space $M_\sigma(\lambda_1+\lambda_2)$ is naturally equipped with a holomorphic symplectic form, constructed as in \cite{Mukai:BundlesK3}. This gives a more conceptual explanation of the existence of the holomorphic symplectic structure.

The case of twisted cubics on $Y$ is even more interesting from many perspectives. Assume that $Y$ does not contain a plane. It was shown in \cite{LLSvS} that the irreducible component $M_3$ of the Hilbert scheme parametrizing twisted cubic curves on $Y$ is a smooth projective variety of dimension ten. Moreover, they proved that the morphism sending $C$ to the three-dimensional projective space spanned by $C$ factorizes through a $\P^2$-fibration $M_3 \rightarrow M_Y'$. Here the variety $M_Y'$, constructed by studying determinantal representations of cubic surfaces in $Y$, is smooth and projective of dimension eight. Finally, they proved that the divisor in $M_Y'$ determined by non CM twisted cubics on $Y$ can be contracted and the resulting variety $M_Y$ is a smooth projective hyperk\"ahler eightfold. In addition, the cubic fourfold $Y$ is contained in $M_Y$ as a Lagrangian submanifold and $M_Y'$ is the blow-up of $M_Y$ in $Y$.   

From the categorical point of view, every twisted cubic curve $C$ in $Y$ has an associated object $F_C'$ in $\Ku(Y)$ with Mukai vector $2\lambda_1+\lambda_2$ (see Section \ref{section:twiscubobj}). Note that the moduli space $M_\sigma(2\lambda_1+\lambda_2)$ is a projective hyperk\"ahler eightfold by \cite{liurendapaper}. Our main result is the following.

\begin{thm}[Theorem \ref{mainThm} and Theorem \ref{thm:mainmoduli}]
\label{thm:maintwistedcubic}
Let $Y$ be a smooth cubic fourfold not containing a plane. If $C$ is a twisted cubic on $Y$, then the object $F_C'$ is $\sigma$-stable. Moreover, the projective hyperk\"ahler eightfold $M_\sigma(2\lambda_1+\lambda_2)$ parametrizes only objects of the form $F_C'$, and it is isomorphic to the LLSvS eightfold $M_Y$.
\end{thm}

\noindent \textbf{Applications.}
As explained in the Appendix of \cite{BLMS:kuzcomponent}, where they treated the case of very general cubic fourfolds, the interpretation of the Fano variety $F_Y$ as a moduli space of stable objects in $\Ku(Y)$ can be used to give a different proof of the categorical version of Torelli Theorem for cubic fourfolds. Thus, Theorem \ref{thm:Fano_main} allows to apply this argument without assumptions on $Y$ (Corollary \ref{cor:Torellicat}). 

A direct consequence of Theorem \ref{thm:maintwistedcubic} is the identification of the period point of $M_Y$ with that of $F_Y$.

\begin{prop}[Proposition \ref{prop:period$M_Y$}]
For a cubic fourfold $Y$ not containing a plane, the period point of $M_Y$ is identified with the period point of the Fano variety $F_Y$.
\end{prop}

A second application of Theorem \ref{thm:maintwistedcubic} is the characterization of when $M_Y$ is birational to a Hilbert scheme of points on a K3 surface (Proposition \ref{prop:eightfoldvsHilb}).

Derived Torelli Theorem has been proved in \cite{Huybrechts:cubicfourfold} for very general cubic fourfolds, for cubic fourfolds with an associated K3 surface and for general cubic fourfolds. Section \ref{subsec:derTor} is an attempt to extend this result for every cubic fourfold. In particular, we show that our strategy works in the simple case of the identity on $\Ku(Y)$, as explained below. 

\begin{prop}[Proposition \ref{prop:Y=Y'}]
Let $Y$ be a cubic fourfold not containing a plane. Then the composition of the projection functor on the Kuznetsov component of $Y$ with the embedding $\Ku(Y) \hookrightarrow \Db(Y)$ is a Fourier-Mukai functor with kernel given by the restriction of the (quasi-)universal family on $M_{\sigma}(2\lambda
_1+\lambda_2) \times Y$ to $Y \times Y$.
\end{prop}

\noindent \textbf{Related works.} The hyperk\"ahler structure on the Fano variety $F_Y$ was firstly observed in \cite{Beauville-cubic4fold}, by a deformation argument. Later in \cite{Kuznetsov2009Symplectic}, another construction was provided using Atiyah classes.

In the case of twisted cubics, the variety $M_Y$ appeared for the first time in the beautiful work \cite{LLSvS}. Their strategy relies on a detailed analysis of the singularities and the determinantal representations of the twisted cubics and the cubic surfaces in $Y$. One feature of our approach is that it only involves homological properties of twisted cubic curves; this simplifies a lot the argument. 

In \cite{LLMS} the authors gave an interpretation of \cite{LLSvS}'s geometric picture in the categorical setting. In particular, they described $M_Y'$ and $M_Y$ as components of moduli spaces of Gieseker stable sheaves on $Y$. For very general cubic fourfolds, they also realized the contraction from $M_Y'$ to $M_Y$ via wall-crossing in tilt-stability.

We point out that Theorem \ref{thm:Fano_main} and Theorem \ref{thm:maintwistedcubic} were proved for very general cubic fourfolds in \cite[Appendix]{BLMS:kuzcomponent} and \cite{LLMS}, respectively. In this situation, the algebraic Mukai lattice of $\Ku(Y)$ is exactly the $A_2$ lattice. This property rules out most of the potential walls, allowing to prove the theorems without going through the construction of the stability conditions. It was made clear in \cite{AddingtonNicolas2014} and \cite{LLMS} that for each twisted cubic $C$, the object $F_C'$ is the correct one to consider.

\noindent \textbf{Update.} Section \ref{sec:hilb} was added in a version of this paper submitted in July 2019. Shortly after that, we learnt that this was proved independently in \cite{BBMP} by a similar method. In \cite{AG} the authors give an independent proof of Proposition \ref{prop:period$M_Y$} using classical techniques. They also show the analogous statement of Proposition \ref{prop:eightfoldvsHilb}.\\
\\
\textbf{Plan of the paper.} In Section \ref{section:Kuznetsov component and stability conditions} we recall the definition of (weak) stability conditions on triangulated categories and the construction of Bridgeland stability conditions on $\Ku(Y)$ in \cite{BLMS:kuzcomponent}. Roughly speaking, they are obtained by tilting a second time the weak stability conditions $\sigma_{\alpha,-1}$ and, then, restricting to $\Ku(Y)$. Finally, we introduce the objects associated to twisted cubics, whose stability is studied in this context. Section \ref{section:Wall crossing and stability} is the main part of the paper. Firstly, we compute walls and the Chern character up to degree two of possible destabilizing objects with respect to $\sigma_{\alpha,-1}$. Secondly, we prove that the first wall can be crossed by preserving stability in the aCM case, while for non CM curves we need to consider the projection of these objects in $\Ku(Y)$ (Proposition \ref{3/4}). In fact, their projection remains stable after every wall, as we show in Section \ref{section:1/4}. Finally, in Section \ref{sec:twisted_final} we prove Theorem \ref{thm:maintwistedcubic}. Section \ref{section:lines} is devoted to the proof of Theorem \ref{thm:Fano_main} and in Section \ref{section:applications} we discuss some applications. 
\\

\noindent \textbf{Acknowledgements.} We are very grateful to Arend Bayer, Mart\'i Lahoz, Emanuele Macr\`i and Paolo Stellari for many useful discussions and conversations.

The first named author is a Leverhulme Early Career Fellow and would like to acknowledge the Leverhulme Trust for the support. The second named author would like to acknowledge the National Group of Algebraic and Geometric Structures and their Applications (GNSAGA-INdAM) for financial support. The third named author was partially supported by the NSF FRG grant DMS-1664215 (PI: Macr\`i).

Part of this paper was written when the third author was visiting the Department of Mathematics of University of Edinburgh, when the second and the third author were visiting Institut Henri Poincar\'e and when the first and the second author were visiting the Department of Mathematics of Northeastern University. It is a pleasure to thank these institutions for their hospitality.\\

\section{Kuznetsov component and stability conditions}
\label{section:Kuznetsov component and stability conditions}

In this section we introduce some notations and results we will use in the rest of the paper. Firstly, we recall some basic definitions about (weak) stability conditions and the construction of stability conditions on the Kuznetsov component of a cubic fourfold, introduced in \cite{BLMS:kuzcomponent}. In particular, we show that these stability conditions do not depend on the line fixed at the very beginning of the construction in \cite{BLMS:kuzcomponent} (see Proposition \ref{prop:change_line}). Finally, we define the objects $F_C'$ associated to twisted cubics which we will study in this work.

\subsection{(Weak) stability conditions}

In this section we briefly recall the definition of (weak) stability conditions for a $\C$-linear triangulated category $\TT$, following the summary in \cite[Section 2]{BLMS:kuzcomponent}. Essentially, a (weak) stability condition is the data of the heart of a bounded t-structure and of a (weak) stability function, satisfying certain conditions.

\begin{defn}
The \emph{heart of a bounded t-structure} is a full subcategory $\AA$ of $\TT$ such that \\
i) for $E$, $F$ in $\AA$ and $n <0$, we have $\Hom(E,F[n])=0$, and \\
ii) for every $E$ in $\TT$, there exists a sequence of morphisms
$$0=E_0 \xrightarrow{\phi_1} E_1 \xrightarrow{\phi_2} \dots \xrightarrow{\phi_{m-1}} E_{m-1} \xrightarrow{\phi_m} E_m=E$$
such that the cone of $\phi_i$ is of the form $A_i[k_i]$, for some sequence $k_1 > k_2 > \dots > k_m$ of integers and $A_i$ in $\AA$. 
\end{defn}

Recall that the heart of a bounded t-structure is an abelian category by \cite{BBD}. 

\begin{defn}
Let $\AA$ be an abelian category. A group homomorphism $Z: K(A) \rightarrow \C$ is a
\emph{weak stability function} (resp.\ a \emph{stability function}) on $\AA$ if, for $E \in \AA$, we have $\Im Z(E) \geq 0$, and in the case that $\Im Z(E) = 0$, we have $\Re Z(E)\leq 0$ (resp.\ $\Re Z(E) < 0$ when $E \neq 0$).
\end{defn}

We denote by $K(\TT)$ the numerical Grothendieck group of $\TT$. Let $\Lambda$ be a finite rank lattice with a surjective homomorphism $v: K(\TT) \twoheadrightarrow \Lambda$.

\begin{defn}
A \emph{weak stability condition} on $\TT$ is the data of a pair $\sigma=(\AA,Z)$, where $\AA$ is the heart of a bounded t-structure on $\TT$ and $Z$ is a weak stability function, satisfying the following properties:\\
i) The composition $K(\AA)=K(\TT) \xrightarrow{v} \Lambda \xrightarrow{Z} \C$ is a weak stability function on $\AA$. We will write $Z(-)$ instead of $Z(v(-))$ for simplicity.

For any $E \in \AA$, the slope with respect to $Z$ is given by
$$\mu_{\sigma}(E)= 
\begin{cases}
-\frac{\Re Z(E)}{\Im Z(E)} & \text{if } \Im Z(E) > 0 \\
+ \infty & \text{otherwise.}
\end{cases}$$
An object $E \in \AA$ is $\sigma$-semistable (resp.\ $\sigma$-stable) if for every proper subobject $F$ of $E$, we have $\mu_{\sigma}(F) \leq \mu_{\sigma}(E)$ (resp.\ $\mu_{\sigma}(F) < \mu_{\sigma}(E/F))$. \\
ii) Any object of $\AA$ has a Harder-Narasimhan filtration with $\sigma$-semistable factors.\\
iii) (Support property) There exists a quadratic form $Q$ on $\Lambda \otimes \R$ such that the restriction of $Q$ to $\ker Z$ is negative definite and $Q(E) \geq 0$ for all $\sigma$-semistable objects $E$ in $\AA$.

In addition, if $Z$ is a stability function, then $\sigma$ is a \emph{Bridgeland stability condition}.
\end{defn}

\subsection{Construction of stability conditions}
\label{section:constructstabonku}

Let $Y$ be a smooth cubic fourfold. The bounded derived category of coherent sheaves on $Y$ admits a semiorthogonal decomposition of the form 
\[\Db(Y)=\langle\Ku(Y),\OO_Y,\OO_Y(H),\OO_Y(2H)\rangle,\]
where $H$ is a hyperplane in $Y$ (see \cite[Corollary 2.6]{Kuz:fourfold}). In this section, we briefly recall the construction of Bridgeland stability conditions on $\Ku(Y)$ introduced in \cite{BLMS:kuzcomponent}. 

The algebraic Mukai lattice of $\Ku(Y)$ was introduced in \cite[Proposition and Definition 9.5]{BLMS:kuzcomponent}. Roughly speaking, it consists of algebraic cohomology classes of $Y$ which are orthogonal to the classes of $\OO_Y$, $\OO_Y(H)$, $\OO_Y(2H)$ with respect to the Euler pairing. This lattice always contains two special classes
\[\lambda_1=[\text{pr}(\OO_L(H))]\;\;\;\;\text{and}\;\;\;\;\lambda_2=[\text{pr}(\OO_L(2H))],\]
where $L$ is a line on $Y$ and $\text{pr}:\Db(Y)\to \Ku(Y)$ is the natural projection functor.

The key idea for the construction of stability conditions on $\Ku(Y)$ is to embed the Kuznetsov component into a ``three dimensional" category, where it is easier to define weak stability conditions by tilting. More concretely, let us fix a line $L \subset Y$ which is not on a plane in $Y$, and we denote by
$$\sigma: \tilde{Y} \rightarrow Y$$
the blow-up of $L$ in $Y$. The projection from $L$ to a disjoint $\P^3$ equips $\tilde{Y}$ with a natural conic fibration structure
$$\pi:\tilde Y \to \P^3.$$
In particular, we have an associated sheaf of Clifford algebras over $\P^3$, whose even part (resp.\ odd part) is denoted by $\BB_0$ (resp.\ $\BB_1$). Let $h$ be the hyperplane class on $\P^3$ and we use the same notation for its pullback to $\tilde{Y}$. We consider the $\BB_0$-bimodules 
$$\BB_{2j}:= \BB_0 \otimes \OO_{\P^3}(jh) \quad \text{and} \quad  \BB_{2j+1}:= \BB_1 \otimes \OO_{\P^3}(jh) \quad \text{for } j \in \Z.$$
By \cite[Proposition 7.7]{BLMS:kuzcomponent}, there is a semiorthogonal decomposition of the form
\begin{equation}
\label{sodP3B0}
\Db(\P^3,\BB_0)=\langle \Psi(\sigma^*\Ku(Y)), \BB_1, \BB_2, \BB_3 \rangle,
\end{equation}
where $\Psi: \Db(\tilde{Y}) \to \Db(\P^3,\BB_0)$ is a fully faithful functor defined by
$$\Psi(-)=\pi_*(- \otimes \OO_{\tilde Y}(h) \otimes \EE[1]).$$
Here $\EE$ is a sheaf of right $\pi^*\BB_0$-modules on $\tilde Y$,  constructed in \cite[Section 7]{BLMS:kuzcomponent}. Use $\text{Forg}: \Db(\P^3,\BB_0) \to \Db(\P^3)$ to denote the forgetful functor, it is known that Forg$(\EE)$ is a vector bundle of rank $2$.

Now the first step is to construct weak stability conditions on the derived category $\Db(\P^3,\BB_0):=\Db(\Coh(\P^3,\BB_0))$, where $\Coh(\P^3,\BB_0)$ is the category of coherent sheaves on $\P^3$ with a right $\BB_0$-modules structure. We remark that the Serre functor on $\Db(\P^3,\BB_0)$ is
$$S(-)=(-) \otimes_{\BB_0} \BB_{-3}[3],$$
as shown in \cite{BLMS:kuzcomponent}.

It turns out that, in order to obtain a suitable Bogomolov inequality for $\Db(\P^3,\BB_0)$, it is necessary to modify the usual Chern character. More precisely, for $\FF\in \Db(\P^3,\BB_0)$, the modified Chern character is defined as
\[\ch_{\BB_0}(\FF)=\ch(\text{Forg}(\FF))(1-\frac{11}{32}l),\]
where $l$ denotes the class of a line in $\P^3$. Moreover, the twisted Chern character is given by 
\[\ch_{\BB_0}^\beta=e^{-\beta h}\ch_{\BB_0}=(\rk,\ch_{\BB_0,1}-\rk \beta h,\ch_{\BB_0,2}-\beta h \cdot \ch_{\BB_0,1}+\rk \frac{\beta^2}{2} h^2,\dots).\]
In the next, we will identify the Chern characters on $\P^3$ with rational numbers. 

One useful property of $\ch_{\BB_0}$ is that its image lattice is spanned by the modified Chern characters of $\lambda_1$, $\lambda_2$ and $\ch_{\BB_0,\leq 2}(\BB_i)$ for $i=1,2,3$. See the proof of \cite[Proposition 9.10]{BLMS:kuzcomponent} for details.

We denote by $\mathrm{Coh}^\beta(\P^3, \BB_0)$ the heart of a bounded t-structure obtained by tilting $\mathrm{Coh}(\P^3, \BB_0)$ with respect to the slope stability at slope $\beta$. Furthermore, the discriminant can be defined as
$$\dlt(\FF)=(\ch_{\BB_0,1}(\FF))^2-2\rk(\FF)\ch_{\BB_0,2}(\FF) = (\ch^{\beta}_{\BB_0,1}(\FF))^2-2\rk(\FF)\ch^{\beta}_{\BB_0,2}(\FF).$$
Having these notations, we can state the following result.

\begin{prop}[\cite{BLMS:kuzcomponent}, Proposition 9.3]\label{prop:constructstabcond}
Given $\alpha>0$ and $\beta\in \R$, the pair $\sigma_{\alpha,\beta}=(\mathrm{Coh}^\beta(\P^3, \BB_0), \zab)$ with
$$\zab(\FF)= i\ch^{\beta}_{\BB_0,1}(\FF)+\frac{1}2\alpha^2\ch^{\beta}_{\BB_0,0}(\FF)-\ch^{\beta}_{\BB_0,2}(\FF)$$
defines a weak stability condition on $\Db(\P^3,\BB_0)$. The quadratic form can be given by the discriminant $\dlt$. In particular, for a $\sigma_{\alpha,\beta}$-semistable object $\FF$, we have
\[\dlt(\FF)\geq 0.\]
\end{prop}

\begin{rmk}
We observe that the last part of Proposition \ref{prop:constructstabcond} follows easily from \cite[Theorem 8.3]{BLMS:kuzcomponent}, arguing as in \cite[Section 3]{BMS:stabCY3s}.
\end{rmk}

We recall that when $\ch^{\beta}_{\BB_0,1}(\mathcal F)\neq 0$ the slope of $\mathcal F$ associated to $\sigma_{\alpha,\beta}$ is defined as
$$\mab(\FF) = \frac{-\Re(\zab(\FF))}{\Im(\zab(\FF))}=\frac{\ch_{\BB_0,2}(\FF)-\frac{1}2(\alpha^2+\beta^2)\rk(\FF)}{\ch_{\BB_0,1}(\FF)-\beta\rk(\FF)}-\beta.$$

The second step is to induce stability conditions on $\Ku(Y)$ from the weak stability conditions on $\Db(\P^3,\BB_0)$. We only sketch this part as details will not be used. We fix $\alpha<\frac{1}{4}$ and $\beta=-1$, and we consider the tilting of $\mathrm{Coh}^{-1}(\P^3, \BB_0)$ with respect to $\mab=0$. This new heart is denoted by $\mathrm{Coh}^0_{\alpha, -1}(\P^3, \BB_0)$. Note that $\Ku(Y)$ embeds into $\Db(\P^3,\BB_0)$. As shown in \cite{BLMS:kuzcomponent}, Section 9, the pair
\begin{equation}
\label{stabcond}
\sigma_\alpha=(\mathrm{Coh}^0_{\alpha, -1}(\P^3, \BB_0)\cap \Ku(Y),-iZ_{\alpha,-1})
\end{equation}
defines a Bridgeland stability condition on $\Ku(Y)$.

One subtle issue is that the Clifford structure and the embedding of $\Ku(Y)$ in $\Db(\P^3,\BB_0)$ depend on the choice of the line $L$ to blow up. However, for the induced stability conditions on the Kuznetsov component, we are able to prove the following result.

\begin{prop}
\label{prop:change_line}
For a fixed $\alpha >0$, the induced stability condition $\sigma_\alpha$ defined in \eqref{stabcond} is independent of the choice of $L$.
\end{prop}
\begin{proof}
For simplicity, we denote the stability condition by the pair
\[\sigma_L=(\AA_L,Z_L).\]
The central charge $Z_L$ factors via $\ch^{\beta}_{\BB_0}$, which is independent of the choice of $L$. We need to show that the heart $\AA_L$ is constant.

Let $F_Y$ be the Fano variety of lines on $Y$. It is shown in \cite[Proposition 30.4]{liurendapaper} that $\sigma_L$ is a family of stability conditions over $F_Y$, satisfying the openness of heart property. In particular, if an object $\FF$ is $\sigma_{L_0}$-semistable for a line $L_0\in F_Y$, then there exists an open set $U_0\subset F_Y$, such that $\FF$ is $\sigma_{L}$-semistable for any line $L\in U_0$.

Now we show that in our case, this implies that $\FF$ is $\sigma_{L}$-semistable for any $L\in F_Y$. If not, assume that there exists a line $L_1$ such that $\FF$ is not $\sigma_{L_1}$-semistable. Then we consider the Harder-Narasimhan filtration of $\FF$ with respect to the slicing of $\sigma_{L_1}$:
\[\FF_1\subset\FF_2\subset...\subset\FF_n=\FF.\]
By our assumption, $\FF_1$ is $\sigma_{L_1}$-semistable, and its phase satisfies $\phi(\FF_1)>\phi(\FF)$.

Using the openness of heart property again, we know that there exists an open set $U_1\subset F_Y$, such that for any $L\in U_1$, $\FF_1$ is $\sigma_{L}$-semistable. In particular, if we take a line $L\in U_0\cap U_1$, then $\FF$ and $\FF_1$ are both $\sigma_{L}$-semistable. Since the central charge is independent of $L$, we still have $\phi(\FF_1)>\phi(\FF)$. On the other hand, by our construction there is a non-trivial morphism $\FF_1\to \FF$, giving a contradiction. This concludes the proof of the statement.
\end{proof}

\subsection{Twisted cubics and objects}
\label{section:twiscubobj}

Let $Y$ be a smooth cubic fourfold not containing a plane. As in \cite{LLMS}, given a twisted cubic curve $C$ contained in a cubic surface $S \subset Y$, we denote by $F_C$ the kernel of the evaluation map
$$\mathrm{H}^0(Y,\II_{C/S}(2H)) \otimes \OO_Y \onto \II_{C/S}(2H),$$
where $\II_{C/S}$ is the ideal sheaf of $C$ in $S$. Let $F_C'$ be the projection of $F_C$ in the Kuznetsov category $\Ku(Y)$. Explicitly, as the projection is the composition of the mutations $\R_{\OO_Y(-H)}\L_{\OO_Y}\L_{\OO_Y(H)}$ (see for example \cite{BLMS:kuzcomponent}, Section 3 for the definitions of mutation functors), it is possible to compute that
$$F_C':= \R_{\OO_Y(-H)}F_C.$$
We recall that by \cite[Lemma 2.3]{LLMS}, if $C$ is an aCM twisted cubic curve, then $F_C$ is in $\Ku(Y)$; in this case, $F_C$ and $F_C'$ are identified. If $C$ is a non CM curve, by the definition of $F_C'$, we have the triangle
$$F_C' \rightarrow F_C \rightarrow \OO_Y(-H)[1] \oplus \OO_Y(-H)[2].$$

Using the notations introduced in the previous section, we set 
$$E_C:=\Psi(\sigma^*F_C) \quad \text{and} \quad E_C':=\Psi(\sigma^*F_C');$$
by \eqref{sodP3B0} we have that $E_C'$ is in $\langle \BB_1, \BB_2, \BB_3 \rangle^{\perp}$. 
Applying $\sigma^*$ and $\Psi$, for non CM curve $C$, we get the triangle
\begin{equation}
\label{mutE_C}
E_C' \rightarrow E_C \rightarrow \BB_{-1}[1] \oplus \BB_{-1}[2];
\end{equation}
here we have used \cite[Proposition 7.7]{BLMS:kuzcomponent}. In particular, we note that
$$\ch^{-1}_{\BB_0,\leq 2}(E_C')=\ch^{-1}_{\BB_0,\leq 2}(E_C)=\ch^{-1}_{\BB_0,\leq 2}(\Psi\sigma^*(2 \lambda_1 + \lambda_2))=(0,6,0).$$

\section{Wall-crossing and stability for twisted cubic curves}
\label{section:Wall crossing and stability}

The aim of this section is to prove Theorem \ref{thm:maintwistedcubic}. Firstly, we compute the walls and the twisted Chern character up to degree two of possible destabilizing objects for $E_C$ with respect to $\sigma_{\alpha,-1}$. Secondly, we characterize semistable objects in the heart with negative rank and zero discriminant. This is fundamental to recover the destabilizing objects by their Chern character. In the third part, we show that, $E_C$ are stable with respect to $\sigma_{\alpha,-1}$ for $\alpha$ large (Proposition \ref{stab_alphalarge}). This gives us the starting point for wall crossing. To cross the first wall, we need to consider the projection $E_C'$ in the Kuznetsov component in order to preserve the stability. Then, we prove that $E_C'$ remains stable after the other walls. Finally, we discuss the stability after the second tilt, and we relate the moduli space which parametrizes these stable objects to the LLSvS eightfold.   

\subsection{Computation of the walls with respect to $\sigma_{\alpha,-1}$}\label{sec:potential_walls}

Having the stability of $E_C$ for $\alpha$ large from Proposition \ref{stab_alphalarge}, we are now interested in computing explicitly the walls where the object could potentially become strictly semistable. In this section, we list the character $\ch^{-1}_{\BB_0,\leq 2}$ of all possible destabilizing objects of $E_C$ and $E_C'$ with respect to the weak stability conditions $\sigma_{\alpha,-1}$. 

We recall that by \cite[Remark 8.4]{BLMS:kuzcomponent}, the rank of $\BB_0$-modules on $\P^3$ is always a multiple of $4$. Thus, we write the characters of the destabilizing subobjects and quotient objects as
\begin{equation}
\label{chF}
(0,6,0)=(4a,b,\frac{c}{8})+(-4a,6-b,-\frac{c}{8})
\end{equation}
for $a, b, c \in \Z$. These characters have to satisfy several additional conditions:
\\
i) The two characters have non-negative discriminant $\dlt$ as recalled in Proposition \ref{prop:constructstabcond}.
\\
ii) There exists $\alpha>0$ such that the two characters have the same slope with respect to $\sigma_{\alpha,-1}$.
\\
iii) The two characters should be integral combinations of the characters of $\lambda_1$ and $\lambda_2$, and $\ch^{-1}_{\BB_0,\leq 2}(\BB_i)$ for $i=1,2,3$.
\\
iv) The ordinary Chern character of objects in $\Db(\P^3)$ truncated to degree $2$ is represented by a triple $(R,C,D/2)$, where $C$ and $D$ are integers of the same parity. Thus, the two characters have the form
$$(R,C,\frac{D}{2})(1,0,-\frac{11}{32})(1,1,\frac{1}{2})=(R,C+R,\frac{D}{2}+C-\frac{5}{16}R).$$ 
These conditions reduce the possible destabilizing characters to finitely many cases, which we list below. The computation is rather elementary and we omit the details.

\begin{prop}
\label{walls}
The possible solutions of \eqref{chF} are:
\begin{enumerate}
\item for $\alpha=3/4$, $a=1$, $b=3$, $c=9$; 
\item for $\alpha=1/4$, 
\begin{enumerate}
\item $a=\pm 1$, $b=1$, $c=\pm 1$;
\item $a=\pm 2$, $b=2$, $c=\pm 2$;
\item $a=\pm 3$, $b=3$, $c=\pm 3$;
\item $a=1$, $b=3$, $c=1$;
\end{enumerate}
\item for $\alpha= 1/12$, $a=9$, $b= 3$, $c= 1$.
\end{enumerate}
\label{prop:potentialwalls}
\end{prop}

Note that the stability condition $\sigma_{\alpha}$ is constructed from $\sigma_{\alpha,-1}$ with $\alpha < 1/4$. In the rest of this section, we will study the stability of $E_C$. We will first prove that if $C$ is an aCM curve, then $E_C$ remains stable with respect to $\sigma_{\alpha,-1}$ after the first wall. On the other hand, if $C$ is non CM, then $E_C$ is destabilized. In particular, we need to consider the mutation $E_C'$ of $E_C$, which instead becomes stable. Then we prove that the second wall can be crossed without changing the stability of $E_C'$. The third wall also does not change the stability of $E_C'$; this fact can be directly proved without using specific information about the destabilizing objects.

\subsection{Stable objects of discriminant zero}

The following general lemma will be crucial in order to study the destabilizing objects by their Chern characters. The basic idea is that a stable object $E$ of zero discriminant and negative rank has to be stable with respect to any weak stability condition $\sigma_{\alpha,\beta}$. Then, comparing the slopes of $E$ and $\BB_i$ with respect to different stability conditions, we get strong restrictions on $\Hom(\BB_i,E[j])$, which can be used to show that $E=\BB_i^{\oplus n}[1]$.

\begin{lem}[Stable objects of discriminant zero]\label{lem:dlt0obj}
Let $E$ be a $\sigma_{\alpha_0,\beta_0}$-semistable object in $\Coh^{\beta_0}(\P^3,\BB_0)$ for some $\alpha_0 > 0$ and $\beta_0 \in \R$. Assume that $\dlt(E)= 0$ and $\rk(E)<0$. Then
$$E=\BB_i^{\oplus n}[1] \quad \text{ for some } i\in \Z \text{ and } n\in\N.$$
\end{lem}
\begin{proof}
In order to simplify the notations, we set
$$\mu_E=\frac{\chbo(E)}{\rk(E)}.$$
As we will compare the slopes of $E$ with $\BB_i$, it is helpful to keep in mind that
$$\frac{\chbo(\BB_i)}{\rk(\BB_i)}=\frac{i}{2}-\frac{1}{4}.$$

Without loss of generality, by considering $E\otimes_{\BB_0}\BB_k$ for suitable $k\in \Z$, we may assume that
$$\mu_E\in [-\frac{1}4,\frac{1}4).$$

By choosing a stable factor of $E$, we may first assume that $E$ is actually $\sigma_{\alpha_0,\beta_0}$-stable. By \cite[Lemma 3.9]{BMS:stabCY3s}, when $\beta> \mu_E-1$, the object $E$ can become strictly semistable only when each stable factor $E_i$ satisfies $\dlt(E_i)<\dlt(E)=0$, which is not possible. Therefore, we deduce that $E$ is $\sab$-stable for $\beta> \mu_E-1$. In particular, we have that $E$ is $\sigma_{0+,\beta_1}$ stable for
$$\mu_E<\beta_1+1<\frac{1}{4},$$
where the notation $\sigma_{0+,\beta_1}$ means that it is possible to find suitable values of $\alpha>0$, realizing this relations between the slopes. We denote the slope function of this stability by $\mu_{0+,\beta_1}$.

Since $\rk(E)<0$, we have (see Figure \ref{fig:delta0obj}): 
\[\mu_{0+,\beta_1}(\BB_{-2}[1])<\mu_{0+,\beta_1}(E)<\mu_{0+,\beta_1}(\BB_1).\]

\begin{figure}[htb]
\tikzset{%
    add/.style args={#1 and #2}{
        to path={%
 ($(\tikztostart)!-#1!(\tikztotarget)$)--($(\tikztotarget)!-#2!(\tikztostart)$)%
  \tikztonodes},add/.default={.2 and .2}}
}
\begin{center}

\begin{imgresize}{1\textwidth}

\begin{tikzpicture}[domain=2:1]

\coordinate (E) at (-0.1,0.005);
\node  at (E) {$\bullet$};
\node [below] at (E) {$E$};

\coordinate (A2) at (-1.25,0.78125);
\node  at (A2) {$\bullet$};
\node [above right] at (A2) {$\mathcal B_{-2}$};
\coordinate (A1) at (-0.75,0.28125);
\node  at (A1) {$\bullet$};
\node [above right] at (A1) {$\mathcal B_{-1}$};
\coordinate (B0) at (-.25,0.03125);
\node  at (B0) {$\bullet$};
\node [above ] at (B0) {$\mathcal B_{0}$};
\coordinate (B1) at (.25,0.03125);
\node  at (B1) {$\bullet$};
\node [above] at (B1) {$\mathcal B_{1}$};
\coordinate (B3) at (1.25,0.78125);
\node  at (B3) {$\bullet$};
\node [above left] at (B3) {$\mathcal B_{3}$};
\coordinate (B2) at (0.75,0.28125);
\node  at (B2) {$\bullet$};
\node [above left] at (B2) {$\mathcal B_{2}$};

\coordinate (BT3) at (-0.5,0.127);
\coordinate (BT3s) at (-0.54,0.127);
\node  at (BT3) {$\bullet$};
\draw[->] (-1,0.15) node [left]  {$\Ker Z_{0+,\beta_3}$} -- (BT3s) ;
\coordinate (BT2) at (-0.2,0.021);
\coordinate (BT2s) at (-0.2,-0.02);
\node  at (BT2) {$\bullet$};
\draw[->] (-0.25,-0.15) node [left]  {$\Ker Z_{0+,\beta'_1}$}--(BT2s) ;
\coordinate (BT1) at (0.03,0.00045);
\coordinate (BT1s) at (0.04,-0.04);
\node  at (BT1) {$\bullet$};
\draw[->] (0.1,-0.25) node [below]  {$\Ker Z_{0+,\beta_1}$} -- (BT1s);

\draw[dashed] (B3) -- (1.25,0) node[below] {$\frac{5}4$};

\draw [dashed] (E) to (BT1);
\draw [add =-1 and 6,thick] (E) to (BT1) node [below]{$\mu_{0+,\beta_1}(E)$};

\draw [dashed] (B1) to (BT1);
\draw [add =-1 and 2.8,thick] (BT1) to (B1) node [right]{$\mu_{0+,\beta_1}(\BB_1)$};

\draw [dashed] (A2) to (BT1);
\draw [add =-1 and 0.7,thick] (A2) to (BT1) node [below]{slope$\sim\mu_{0+,\beta_1}(\BB_{-2}[1])$};

\draw[->] [opacity=1] (-1.5,0) -- (,0) node[above right] {}-- (1.5,0) node[below, opacity =1] {$\frac{\ch^{-1}_{\BB_0,1}}{\rk}$};

\draw[->][opacity=1] (0,-0.5)-- (0,0) node [below left] {} --  (0,1) node[above, opacity=1] {$\frac{\ch^{-1}_{\BB_0,2}}{\rk}$};

\draw [thick](-1.3,0.845) parabola bend (0,0) (1.3,0.845) node [above right, opacity =1] {$\Delta_{\BB_0}=0$};
\end{tikzpicture}

\end{imgresize}

\end{center}
\caption{Comparing slopes with respect to $\mu_{0+,\beta_1}$.} \label{fig:delta0obj}
\end{figure}

By comparing the slope and applying Serre duality, it follows that
\[\Hom(\BB_1,E[j])=0\]
for $j\neq 1$. Therefore, $\chi(\BB_1,E)\leq 0$.\\

Now we study the vertical wall. Suppose that $E$ is strictly semistable when $\beta_2=\mu_E-1$. Then each stable factor $E_i$ satisfies one of the two conditions:
$$\rk(E_i)<0 \text{  or  } \chbl(E_i)=(0,0,0).$$
We study these two cases separately. Given a stable factor $E_i$ with negative rank, by \cite[Lemma 3.9]{BMS:stabCY3s} we have that $E_i[-1]$ is in the heart $\Coh^{\beta}(\P^3,\BB_0)$ and it is $\sab$-stable for any $\beta+1<\mu_E$. In particular, $E_i[-1]$ is $\sigma_{0+,\beta_3}$-stable for
$$-\frac{3}{4}<\beta_3+1<\mu_E.$$
Since $\rk(E_i)< 0$, we can compute (see also Figure \ref{fig:delta0obj}\footnote{Instead of computing the $\mu_{0+,\beta_1}$ for each object explicitly, one may also compare the slopes of them by visualizing them on the figure. A point above the parabola represents the kernel of the central charge, while a point below the parabola represents a stable character. We refer to \cite[Section 1.5]{Chunyi-Xiaolei:birational} for details of this setup.}):
\[\mu_{0+,\beta_3}(\BB_{-2}[1])<\mu_{0+,\beta_3}(E_i[-1])<\mu_{0+,\beta_3}(\BB_{1}).\]
As a consequence, we get
$$\Hom (\BB_1,E_i[1])=\Hom(E_i[-1],\BB_{-2}[1])^*=0.$$
Since $E_i$ is also $\sab$-stable for $\beta> \mu_E-1$ (by the same argument used for $E$), we deduce that $\Hom(\BB_1,E_i[j])=0$ for any $j\in \Z$, in other words, $E_i\in \BB_1^\perp$. In particular, $\chi(\BB_1,E_i)=0$. 

In the second case, we show that such a torsion stable factor cannot exist. Assume that $E_i$ is a stable factor with $\chbl(E_i)=(0,0,0)$; note that
$$\Hom_{\BB_0}(\BB_1,E_i[j])=\Hom_{\OO_{\P^3}}(\OO_{\P^3},\Forg(E_i\otimes_{\BB_0}\BB_{-1})[j])=0$$
if and only if $j\neq 0$. This implies that $\chi(\BB_1,E_i)>0$. Since $\chi(\BB_1,E_i)$ is also non positive by the previous computation, we conclude that $E_i$ has to be zero. Hence, we may assume that each stable factor $E_i$ satisfies $\rk(E_i)<0$.\\

Now we want to show that $\chbl(E_i[-1])=c\chbl(\BB_0)$ for some positive integer $c$. It suffices to show that $\frac{\chbo(E_i)}{\rk(E_i)}= -\frac{1}4$. Assume not, we may consider the tilt stability condition $\sigma_{0+,\beta'_1}$ for some
$$\frac{\chbo(\BB_{0})}{\rk(\BB_{0})}<\beta'_1+1<\frac{\chbo(E_i)}{\rk(E_i)}.$$
In this case, we have
\begin{align*}
\mu_{0+,\beta'_1}(\BB_{-1}[1])<\mu_{0+,\beta'_1}(\BB_{0}[1])<\mu_{0+,\beta'_1}(E_i[-1])<\mu_{0+,\beta'_1}(\BB_{2})<\mu_{0+,\beta'_1}(\BB_{3})
\end{align*}
and
\begin{align*}
\mu_{0+,\beta_1}(\BB_{-1}[1])<\mu_{0+,\beta_1}(\BB_{0}[1])<\mu_{0+,\beta_1}(E_i)<\mu_{0+,\beta_1}(\BB_2)<\mu_{0+,\beta_1}(\BB_3).
\end{align*}
Hence
$$\Hom(\BB_2,E_i[j])=\Hom(\BB_3,E_i[j])=0$$
for any $j\in \Z$. This shows that $E_i$ belongs to $\Psi(\sigma^*\Ku(Y))$. In particular,the twisted Chern character of $E_i$ satisfies
$$\chbl(E_i)=a\lambda_1+b\lambda_2$$
for some $(a,b)\neq (0,0)$. Note that any $E_i$ with such truncated twisted Chern character satisfies $\dlt(E_i)\geq 7$. This leads to a contradiction with the assumption that $E$ has zero discriminant. 

We may now assume that $\chbl(E_i[-1])=c\chbl(\BB_0)$ for some positive integer $c$. Since
\begin{align*}
\mu_{0+,\beta_3}(\BB_{-3}[1])<\mu_{0+,\beta_3}(\BB_{-1}[1])<\mu_{0+,\beta_3}(E_i[-1])<\mu_{0+,\beta_3}(\BB_{2}) \\
\end{align*}
and
\begin{align*}
\mu_{0+,\beta_1}(\BB_{-1}[1])<\mu_{0+,\beta_1}(E_i)<\mu_{0+,\beta_1}(\BB_2),
\end{align*}
we have the vanishing $\Hom(\BB_2,E_i[j])=0$ for any $j\in \Z$, and $\Hom(\BB_0,E_i[j])=0$ for any $j\neq 0$ or $-1$. Therefore, we have that 
\begin{align*}
0 & =\chi(\BB_2,E_i)=\chi_{\OO_{\P^3}}(\OO_{\P^3},\Forg(E_i)(-h)) \\
&  =\ch_3(\Forg(E_i)(-h))+2\ch_2(\Forg(E_i)(-h))+\frac{11}6\ch_1(\Forg(E_i)(-h))+\rk(E_i)\\
& =\chi_{\OO_{\P^3}}(\OO_{\P^3},\Forg(E_i))-\ch_2(\Forg(E_i))-\frac{3}2\ch_1(\Forg(E_i))-\rk(E_i)\\
& = \chi_{\OO_{\P^3}}(\OO_{\P^3},\Forg(E_i))-\chbt(E_i)-\frac{1}2\chbo(E_i)-\frac{13}{16}\rk(E_i)\\
& = \chi_{\OO_{\P^3}}(\OO_{\P^3},\Forg(E_i)) - \frac{1}{32}\rk(E_i) +\frac{1}8\rk(E_i)-\frac{13}{16}\rk(E_i)\\ & >\chi_{\OO_{\P^3}}(\OO_{\P^3},\Forg(E_i)) = -\hom(\BB_0,E_i[-1])+\hom(\BB_0,E_i).
\end{align*}
In particular, it follows that $\Hom(\BB_0,E_i[-1])\neq 0$.
As both $\BB_0$ and $E_i[-1]$ are $\sigma_{0+,\beta_3}$-stable with the same slope, we must have $E_i=\BB_0[1]$. Since this condition holds for every stable factor and $\Ext^1(\BB_0,\BB_0)=0$, we deduce that $E=\BB_0^{\oplus n}[1]$ as desired.
\end{proof}

\subsection{First wall: $\alpha=\frac{3}{4}$}
Let us come back to twisted cubic curves in $Y$. Since $Y$ does not contain a plane, it follows that the cubic surface $S$, which is cut out by the $\P^3$ spanned by $C$, is irreducible. We will assume that the line $L$, which is blown up in the cubic fourfold, is disjoint from this $\P^3$. For such a choice of $L$, the blow-up $\sigma$ and the projection $\pi$ map $S$ isomorphically to a cubic surface $S'$ in the base $\P^3$. In this section and next section, for a fixed twisted cubic $C$, we will work with such a line $L$. By Proposition \ref{prop:change_line}, this will not change the stability condition induced on $\Ku(Y)$.

Let $\sigma_{\alpha,\beta}$ be the weak stability condition on $\Db(\P^3,\BB_0)$ introduced in Proposition \ref{prop:constructstabcond}. In the next proposition we prove that $E_C$ is $\sigma_{\alpha,-1}$-stable for $\alpha$ large enough.

\begin{prop}
\label{stab_alphalarge}
The torsion sheaf $E_C$ on $\P^3$ is slope-stable. In particular, $E_C$ is $\sigma_{\alpha,-1}$-stable for $\alpha \gg 0$.
\end{prop}
\begin{proof}
Now we want to compute $E_C$ with respect to $L$. We have
\[E_C=\Psi\sigma^*F_C=\pi_*(\sigma^*F_C \otimes \OO_{\tilde Y}(h) \otimes \EE[1])=\pi_*(\sigma^*\II_{C/S}(2H) \otimes \OO_{\tilde Y}(h) \otimes \EE),\]
where the first two equalities follow from the definitions. The last equality is a consequence of applying the functor $\Psi\sigma^*$ to the defining sequence
$$0 \to F_C \to \mathrm{H}^0(Y,\II_{C/S}(2H)) \otimes \OO_Y \to \II_{C/S}(2H) \to 0,$$
and the fact that $\Psi\sigma^*\OO_Y=0$. As a consequence, the sheaf $E_C$ is torsion free, supported over the irreducible cubic surface $S'$ in $\P^3$. 

Note that $\ch^{-1}_{\BB_0,\leq 2}(E_C)=(0,6,0)$. Let $F$ be a torsion sheaf destabilizing $E_C$. Then we have that $F$ has the same support of $E_C$ and it has rank one as a sheaf over $S'$. It follows that $\ch^{-1}_{\BB_0,\leq 1}(F)=(0,3)$. However, such an object cannot exist in $\Coh(\P^3,\BB_0)$, because this character is not in the lattice spanned by the characters of $\lambda_1$, $\lambda_2$ and $\BB_i$ for $i=1,2,3$. It follows that $E_C$ is slope-stable, in the sense that any proper $\BB_0$-subsheaf of $E_C$ has a smaller slope $\ch_{\BB_0,1}^{-1} / \rk$. Since for $\alpha \rightarrow \infty$, the weak stability $\sigma_{\alpha,-1}$ converges to the slope stability, we deduce the desired statement.
\end{proof}

By Proposition \ref{stab_alphalarge} and Proposition \ref{walls}, we have that $E_C$ is $\sigma_{\alpha,-1}$-stable for $\alpha > 3/4$. In this section, we study the stability of $E_C$ after the first wall.

\begin{prop}[First wall for twisted cubics]
\label{3/4}
For $1/4 < \alpha < 3/4$, we have that $E_C'$ is $\sigma_{\alpha,-1}$-stable.  More precisely:
\begin{itemize}
\item If $C$ is an aCM twisted cubic curve in $Y$, then $E_C'=E_C$ is $\sigma_{\alpha,-1}$-stable.
\item If $C$ is a non CM cubic curve, then $E_C$ becomes strictly $\sigma_{\alpha,-1}$-semistable at the wall $\alpha=3/4$. Instead, for $1/4 < \alpha < 3/4$, the object $E_C'$ is $\sigma_{\alpha,-1}$-stable.
\end{itemize}
\end{prop}
\begin{proof}
Let us consider the destabilizing quotient object given by Proposition \ref{walls} with
$$\ch_{\BB_0,\leq 2}^{-1}=(-4,3,-9/8).$$
By Lemma \ref{lem:dlt0obj}, we know that this object is $\BB_{-1}[1]$. Recall that the Serre functor on $\Db(\P^3,\BB_0)$ is
$$S(-)=(-) \otimes_{\BB_0} \BB_{-3}[3],$$
by \eqref{sodP3B0} we have that
$$\Hom_{\BB_0}(E'_C,\BB_{-1}[1])=\Hom_{\BB_0}(\BB_2,E'_C[2])^{\vee}=0.$$
The first claim follows easily from the fact that $E_C \cong E_C'$ in the aCM case.

Assume now that $C$ is a non CM twisted cubic curve. Then using the sequence \eqref{mutE_C} and the fact that
$$\Hom_{\BB_0}(\BB_2,\BB_{-1}[3])=\Hom_{\BB_0}(\BB_{-1},\BB_{-1})^{\vee}  \cong \C,$$
we get
$$\Hom_{\BB_0}(E_C,\BB_{-1}[1]) \cong \C.$$
In particular, for $\alpha=3/4$, it follows that $E_C$ is strictly $\sigma_{\alpha,-1}$-semistable and the Jordan-H\"{o}lder filtration in $\Coh^{-1}(\P^3,\BB_0)$ is given by
$$0 \rightarrow M_C \rightarrow E_C \rightarrow \BB_{-1}[1] \rightarrow 0.$$
Finally, for $1/4 < \alpha < 3/4$, using again the sequence \eqref{mutE_C}, it is easy to see that the new stable object is $E_C'$, which fits into the sequence
$$0 \rightarrow \BB_{-1}[1]  \rightarrow E_C' \rightarrow M_C\rightarrow 0.$$
\end{proof}

\subsection{Second wall: $\alpha=\frac{1}{4}$}
\label{section:1/4}

The aim of this section is to prove the following result.

\begin{prop}
Let $0 < \alpha < 1/4$. If $C$ is a twisted cubic curve in $Y$, then $E_C'$ is $\sigma_{\alpha,-1}$-stable.
\end{prop}
\noindent This proposition is a consequence of Lemma \ref{lemma:secondwallforcubic} and Lemma \ref{lemma:linecubicinKu} below. 

We firstly consider the objects given by the second part of Proposition \ref{walls} and we show that they cannot destabilize $E_C'$. The key observation is that if $E_C'$ is destabilized, then a slope comparison argument implies that its stable factors have to be in $\Psi(\sigma^*\Ku(Y))$. This will lead to a contradiction, as such stable factors do not exist for the wall $\alpha=1/4$.

\begin{lem}[Second wall for twisted cubics]\label{lemma:secondwallforcubic}
Let $E$ be a $\sigma_{\frac{1}4+\epsilon,-1}$-stable object in $\Psi(\sigma^*\Ku(Y))$ with $\chbl(E)=(0,6,0)$. Then $E$ is $\sigma_{\frac{1}4-\epsilon,-1}$-stable.
\end{lem}
\begin{proof}
Suppose that $E$ is not $\sigma_{\frac{1}4-\epsilon,-1}$-stable; we consider the Harder-Narasimham filtration of $E$ with respect to $\sigma_{\frac{1}4-\epsilon,-1}$:
\[0\rightarrow E_1\rightarrow \dots \rightarrow E_k=E.\]
Here each factor $E_{i+1}/E_i$ is $\sigma_{\frac{1}4-\epsilon,-1}$-semistable with strictly decreasing slopes. 

Assume that Hom$(E_k/E_{k-1},\BB_0[1])\neq 0$. Note that $E_k/E_{k-1}$ is a quotient object of $E$ in the heart $\Coh^{-1}(\P^3,\BB_0)$. Since $\BB_0[1]$ is also an object in $\Coh^{-1}(\P^3,\BB_0)$, the assumption above implies 
$$\Hom(E,\BB_0[1])\neq 0.$$
By Serre duality, we obtain
$$\Hom(\BB_3,E[2])=(\Hom(E,\BB_0[1]))^*\neq 0,$$
which contradicts the condition that $E\in \Psi(\sigma^*\Ku(Y))$. Therefore, it follows that
$$\Hom(\BB_3,E_k/E_{k-1}[2])= 0.$$
By a similar argument, we get
$$\Hom(\BB_1,E_{k-1})= 0.$$

Note that we have the following inequalities: (see Figure \ref{fig:slopeat14})
\begin{align*}
& \mu_{\frac{1}4,-1}(\BB_{-2}[1])<\mu_{\frac{1}4,-1}(\BB_{-1}[1])<\mu_{\frac{1}4,-1}(E_k/E_{k-1})=\\ & \mu_{\frac{1}4,-1}(E_{k-1})< 
 \mu_{\frac{1}4,-1}(\BB_2)<\mu_{\frac{1}4,-1}(\BB_3); \\
 & \mu_{\frac{1}4-\epsilon,-1}(E_k/E_{k-1}) < \mu_{\frac{1}4-\epsilon,-1}(\BB_1)\text{; } \\ 
& \mu_{\frac{1}4-\epsilon,-1}(\BB_{0}[1]) < \mu_{\frac{1}4-\epsilon,-1}(E_i/E_{i-1}) \text{ for every } 1\leq i<k.
\end{align*}

\begin{figure}[htb]
\tikzset{%
    add/.style args={#1 and #2}{
        to path={%
 ($(\tikztostart)!-#1!(\tikztotarget)$)--($(\tikztotarget)!-#2!(\tikztostart)$)%
  \tikztonodes},add/.default={.2 and .2}}
}
\begin{center}

\begin{imgresize}{1\textwidth}

\begin{tikzpicture}[domain=1:1]
\newcommand{\hts}{2}
\coordinate (A1) at (-0.75,0.28125*\hts);
\node  at (A1) {$\bullet$};
\node [above right] at (A1) {$\mathcal B_{-1}$};
\coordinate (B0) at (-.25,0.03125*\hts);
\node  at (B0) {$\bullet$};
\node [above right] at (B0) {$\mathcal B_{0}$};
\coordinate (B1) at (.25,0.03125*\hts);
\node  at (B1) {$\bullet$};
\node [above] at (B1) {$\mathcal B_{1}$};
\coordinate (B2) at (0.75,0.28125*\hts);
\node  at (B2) {$\bullet$};
\node [above left] at (B2) {$\mathcal B_{2}$};

\coordinate (K1) at (.45,0.03125*\hts);
\node  at (K1) {$\bullet$};
\node [above] at (K1) {$E_i/E_{i-1}$};
\coordinate (K2) at (-0.5,0.03125*\hts);
\node  at (K2) {$\bullet$};
\node [above left] at (K2) {$E_k/E_{k-1}$};

\coordinate (BT1) at (0,0.01*\hts);
\node  at (BT1) {$\bullet$};
\draw[->] (BT1) -- (0.25,-0.25*\hts) node [below]  {$\Ker Z_{\frac{1}4-\epsilon,-1}$};

\coordinate (BT0) at (0,0.03125*\hts);
\coordinate (BT0s) at (-0.01,0.03125*\hts+0.015);
\node  at (BT0) {$\bullet$};
\draw[->] (-0.25,0.25*\hts) node [above]  {$\Ker Z_{\frac{1}4,-1}$} --(BT0s);

\draw[dashed] (A1) -- (-0.75,0*\hts) node[below] {$-\frac{3}4$};

\draw [dashed] (B0) to (BT1);
\draw [add =-1 and 1.8,thick] (B0) to (BT1) node [below left]{$\mu_-(\BB_0[1])$};

\draw [dashed] (K2) to (BT1);
\draw [add =-1 and 1.3,thick] (K2) to (BT1) node [below]{$\mu_-(E_k/E_{k-1})$};

\draw [add =0 and 1.8,thick] (BT1) to (B1) node [above]{$\mu_-(\BB_1)$};

\draw [add =0 and 1,thick] (BT1) to (K1) node [above]{$\mu_-(E_i/E_{i-1})$};

\draw [add =1 and 1,dashed] (B0) to (B1) node [right]{$\mu_{\frac{1}4,-1}(E)$};

\draw [dashed] (A1) to (BT0);
\draw [add =-1 and 1,thick] (A1) to (BT0) node [below]{$\mu_{\frac{1}4,-1}(\BB_{-1}[1])$};

\draw[->] [opacity=1] (-0.8,0) -- (0.8,0);

\draw[->][opacity=1] (0,-0.3*\hts)-- (0,0) node [below left] {} --  (0,0.4*\hts) node[above, opacity=1] {$\frac{\ch^{-1}_{\BB_0,2}}{\rk}$};

\draw [thick](-0.8,0.32*\hts) parabola bend (0,0) (0.8,0.32*\hts) node [above right, opacity =1] {$\Delta_{\BB_0}=0$};

\end{tikzpicture}

\end{imgresize}

\end{center}
\caption{Comparing the slopes of $\BB_j$ and $E_i/E_{i-1}$ at $\mu_{\frac{1}4,-1}$ and $\mu_-=\mu_{\frac{1}4-\epsilon,-1}$} \label{fig:slopeat14}
\end{figure}
Note that $E_{k-1}$ is semistable at a closed subset on the space of tilt stability conditions, we may choose $\epsilon$ small enough so that $E_{k-1}$ is $\mu_{\frac{1}4,-1}$-semistable, and each $E_i/E_{i-1}$ is $\mu_{\frac{1}4-\epsilon,-1}$-semistable. By Serre duality, we have
\[\Hom(\BB_s,E_k/E_{k-1}[j])=\Hom(\BB_s,E_{k-1}[j])=0,\]
for $s=1,2,3$ and every $j\neq 1$. Here $\Hom(\BB_3,E_{k-1}[2])=\Hom(E_{k-1},\BB_0[1])=0$ since we have $\Hom(\BB_3,E_{i}/E_{i-1}[2])=\Hom(E_{i}/E_{i-1},\BB_0[1])=0$ for every factor.  Since $E\in \Psi(\sigma^*\Ku(Y))$, $$\chi(\BB_s,E_k/E_{k-1})+\chi(\BB_s,E_{k-1})=\chi(\BB_s,E)=0$$ for $s=1,2,3$. Therefore, 
\[\Hom(\BB_s,E_k/E_{k-1}[1])=\Hom(\BB_s,E_{k-1}[1])=0,\]
for $s=1,2,3$. In particular, we deduce that $E_{k-1}$ and $E_k/E_{k-1}$ are in $\Psi(\sigma^*\Ku(Y))$. As a consequence, the twisted Chern character of $E_{k-1}$ satisfies
\begin{align*}
 \chbl(E_{k-1}) =c\lambda_1+d\lambda_2\in\{(x,y,-\frac{7}{32}x)\}.
\end{align*}
By the classification of potential destabilizing factors as that in Proposition \ref{prop:potentialwalls} (2), we also have the character of $E_{k-1}$ should of the form:
$$ \chbl(E_{k-1}) =a\chbl(\BB_1)+b(0,6,0) \in \{(x,y,\frac{1}{32}x)\} $$
We conclude that $\chbl(E_{k-1})$ must be of the form $(0,y,0)$. However, it would destabilize $E$ with respect to $\sigma_{\frac{1}4+\epsilon,-1}$, which is a contradiction. This proves the stability of $E$ as in the statement.
\end{proof}

Now we consider the third wall in Proposition \ref{walls}. In this case, we obtain a slightly general result, showing that for $\alpha<1/4$, the only stable objects are in $\Psi(\sigma^*\Ku(Y))$ and they cannot be destabilized. The argument is similar to the proof of Lemma \ref{lemma:secondwallforcubic}.

\begin{lem}[After the second wall]\label{lemma:linecubicinKu}
For $0< \alpha_0<\frac{1}4$, let $E$ be a $\sigma_{\alpha_0,-1}$ stable object such that $[E]=[E'_C]$ in the numerical Grothendieck group. Then $E$ is in $\Psi(\sigma^*\Ku(Y))$ and it is $\sigma_{\alpha,-1}$ stable for any $0< \alpha\leq \alpha_0$. 
\end{lem}
\begin{proof}
We set $\mu=\mu_{\alpha_0,-1}$ for simplicity. As $[E]=[E_C']$ in the numerical Grothendieck group, we observe that 
\[ \mu(\BB_{-2}[1])<\mu(\BB_{-1}[1])<\mu(\BB_{0}[1])<\mu(E)<
\mu(\BB_1)< \mu(\BB_2)<\mu(\BB_3).\]
By Serre duality we have that $$\Hom(\BB_s,E[j])=0$$ for any $s=1,2,3$ and $j\neq 1$. Again, since $[E]=[E_C']$ in the numerical Grothendieck group, we have $$\chi(\BB_s,E)= \chi(\BB_s,E'_C)=0$$ for $s=1,2,3$. It follows that $$\Hom(\BB_s,E[1])=0$$ for any $s=1,2,3$, proving that $E$ belongs to $\Psi(\sigma^*\Ku(Y))$.

Suppose that $E$ becomes strictly $\sigma_{\alpha,-1}$-semistable for some $\alpha<\alpha_0<\frac{1}4$. We may consider the Harder-Narasimhan filtration of $E$ with respect to  $\sigma_{\alpha-\epsilon,-1}$:
\[0\subset E_1\subset \dots \subset E_k=E.\]
By comparing $\mu_{\alpha-\epsilon,-1}$ of $E_k/E_{k-1}$, $E_{k-1}$, $\BB_{-2}[1]$, $\BB_{-1}[1]$, $\BB_{0}[1]$, $\BB_1$, $\BB_2$ and $\BB_3$, using the same argument applied in the proof of Lemma \ref{lemma:secondwallforcubic}, we get the conclusion that both $E_k/E_{k-1}$ and $E_{k-1}$ are in $\Psi(\sigma^*\Ku(Y))$. But this implies that $$\chbl(E_{k-1})=a\lambda_1+b\lambda_2\in \{(x,y, -\frac{7}{32}x)\}.$$ 
Note that $\mu_{\alpha,-1}(E_{k-1})=\mu_{\alpha,-1}(E)$, we have $\chbl(E_{k-1})\in\{(x,y, \frac{1}{2}\alpha^2x)\}$. Hence, we must have $\chbl(E_{k-1})=(0,y,0)$, which leads to a contradiction. This proves the stability of $E$ as we wanted.
\end{proof}

\subsection{Stability after second tilt and the moduli space}
\label{sec:twisted_final}

This section is devoted to the proof of Theorem \ref{thm:maintwistedcubic}. Firstly, we show that $E_C'$ is $\sigma_{\alpha,-1}^0$-stable, where $\sigma_{\alpha,-1}^0$ is the weak stability condition on $\Db(\P^3,\BB_0)$ obtained by tilting $\sigma_{\alpha,-1}$ (see \cite{BLMS:kuzcomponent}, the proof of Theorem 1.2). In particular, this implies the stability of $F_C'$ with respect to the stability condition $\sigma:=\sigma_{\alpha}$ on $\Ku(Y)$, defined in \eqref{stabcond} and constructed in \cite{BLMS:kuzcomponent}. 

\begin{thm}
\label{mainThm}
Let $Y$ be a smooth cubic fourfold not containing a plane. If $C$ is a twisted cubic curve on $Y$, then the object $F_C'$ is $\sigma$-stable, with respect to $\sigma:=\sigma_{\alpha}$ given in \eqref{stabcond}.
\end{thm}

\begin{proof}
Note that by definition the stability function for $\sigma_{\alpha,-1}^0$ is $Z_{\alpha,-1}$ multiplied by $-\sqrt{-1}$. In particular, the new heart obtained through the second tilt is just the previous heart rotated by ninety degrees. It follows that the walls would correspond to those we have computed for $\sigma_{\alpha,-1}$ and the previous argument proves that these can be crossed preserving the stability of $E_C'$. This implies the stability of $E_C'$ with respect to $\sigma_{\alpha,-1}^0$. As the stability conditions $\sigma$ on $\Ku(Y)$ are induced from $\sigma_{\alpha,-1}^0$ for $\alpha<1/4$, and $F_C'$ is in the Kuznetsov component, we get the desired statement.
\end{proof}

Now we are able to describe the moduli space $M_\sigma(2\lambda_1+\lambda_2)$ of $\sigma$-stable objects with Mukai vector $2\lambda_1+\lambda_2$ and, in particular, its identification with the LLSvS eightfold $M_Y$ constructed in \cite{LLSvS}. We use a standard argument, which is very similar to \cite[Section 5.3]{LLMS}. We point out that the results in \cite{liurendapaper} implies that $M_\sigma(2\lambda_1+\lambda_2)$ is a smooth, projective, irreducible hyperk\"ahler eightfold.

\begin{thm}
\label{thm:mainmoduli}
The moduli space $M_\sigma(2\lambda_1+\lambda_2)$ parametrizes only objects of the form $F_C'$. Moreover, $M_\sigma(2\lambda_1+\lambda_2)$ is isomorphic to the LLSvS eightfold $M_Y$.
\end{thm}

\begin{proof}
Let $M_3$ be the irreducible component of the Hilbert scheme parameterizing twisted cubic curves on $Y$. Then there exists a quasi-universal family $\FF$ on $Y \times M_3$ parameterizing the sheaves $\II_{C/Y}(2H)$. By \cite[Theorem 5.8]{Kuznetsov2011Base_change}, we have a semiorthogonal decomposition of the form
\[\Db(Y \times M_3)=\langle\Ku(Y \times M_3),\OO_Y\boxtimes\Db(M_3), \OO_Y(H)\boxtimes\Db(M_3),\OO_Y(2H)\boxtimes\Db(M_3) \rangle.\]
Now consider the relative projection $\FF'$ of $\FF$ in $\Ku(Y \times M_3)$. As in \cite{AddingtonNicolas2014}, it is possible to verify that the projection of $\II_{C/Y}(2H)$ in the Kuznetsov component is exactly $F_C'$ (see Section \ref{subsec:derTor} for the computation in the non CM case). So, Theorem \ref{mainThm} implies that $\FF'$ is a quasi-universal family of $\sigma$-stable objects $F_C'$ in $\Ku(Y)$. Then there is an induced dominant morphism $M_3 \to M_\sigma(2\lambda_1+\lambda_2)$. As $M_3$ is projective, we know that this morphism is surjective. This concludes the first statement.

For the second statement, we just need to show that for two twisted cubic curves $C_1$ and $C_2$, we have $F_{C_1}'=F_{C_2}'$ if and only if $C_1$ and $C_2$ are contained in the same fiber of the morphism $M_3\to M_Y$ constructed in \cite{LLSvS}. This is exactly proved in \cite[Proposition 2]{AddingtonNicolas2014}. Indeed, they consider the projection in the K3 subcategory $\langle \OO_Y(-H), \OO_Y, \OO_Y(H) \rangle^{\perp}$, which is equivalent to $\Ku(Y)$. This ends the proof of the theorem.
\end{proof}

\section{Fano varieties of lines and stability conditions}
\label{section:lines}

In this section, we use a similar argument to that applied in the case of twisted cubic curves in order to describe the Fano variety $F_Y$ parametrizing lines in a cubic fourfold $Y$ as a moduli space of Bridgeland stable objects. 

Recall that given a line $\ell$ in $Y$, we can associate an object $P_\ell$ in $\Ku(Y)$, which sits in the distinguished triangle
\[\OO_Y(-H)[1] \to P_\ell \to \II_\ell,\]
where $\II_\ell$ denotes the ideal sheaf of $\ell$ in $Y$ (see \cite[Section 6.3]{LLMS}). It is easy to compute that the Mukai vector of $P_\ell$ is $\lambda_1+\lambda_2$.

By Proposition \ref{prop:change_line}, we can assume that the line $L$ used in the construction of stability conditions is disjoint from $\ell$. Let us compute explicitly the image $M_\ell=\Psi(\sigma^*P_\ell)$ in $\Db(\P^3,\BB_0)$. By \cite[Proposition 7.7]{BLMS:kuzcomponent}, we have that
\[\Psi(\OO_{\tilde Y}(-H))=\BB_{-1}.\]
On the other hand, we consider the sequence
\[\II_\ell\to \OO_Y \to \OO_\ell.\]
We recall that
\[\Psi(\OO_{\tilde Y})=0.\]
By our assumption, we know that $\ell$ maps isomorphically to a line in $\P^3$; hence we have that
\[\Psi(\sigma^*\II_\ell)=\Psi(\sigma^*\OO_\ell)[-1]=\pi_*(\EE(h)|_{\sigma^{-1}(\ell)})\]
is a torsion sheaf supported over the image of $\ell$ in $\P^3$. We denote it by $\EE_\ell$. So we have the distinguished triangle 
\begin{equation}
\label{sequenceMl}
\BB_{-1}[1]\to M_\ell \to \EE_\ell
\end{equation}
in $\Db(\P^3,\BB_0)$.

Note that
$$\ch_{\BB_0,\leq 2}^{-1}(M_\ell)=(-4,3,\frac{7}{8}).$$
The following lemma gives us the starting point of the wall crossing argument.
\begin{lem}
\label{lemma:stablineontop}
The object $M_\ell$ is $\sigma_{\alpha,-1}$-stable for $\alpha\gg 0$.
\end{lem}
\begin{proof}
Assume that $M_\ell$ is not stable with respect to $\sigma_{\alpha,-1}$ for $\alpha\gg 0$. Then there is a destabilizing sequence of $M_\ell$
$$P \rightarrow M_\ell \rightarrow Q$$
in the heart $\Coh^{-1}(\P^3,\BB_0)$, where $P$, $Q$ are $\sigma_{\alpha,-1}$-semistable for $\alpha\gg 0$, and $\mu_{\alpha,-1}(P) > \mu_{\alpha,-1}(Q)$. We have two possibilities for $P$: either it is torsion or it has rank equal to $-4$. If we are in the first case, then, for $\alpha$ going to infinity, the slope $\mu_{\alpha,-1}(P)$ is a finite number, while $\mu_{\alpha,-1}(Q)=+\infty$. Thus such a $P$ cannot destabilize $M_{\ell}$.

In the case $\rk(P)=-4$, let us consider the cohomology sequence
$$0 \rightarrow \HH^{-1}(P) \rightarrow \HH^{-1}(M_\ell) \rightarrow \HH^{-1}(Q) \rightarrow \HH^{0}(P) \rightarrow \HH^{0}(M_\ell) \rightarrow \HH^{0}(Q) \rightarrow 0.$$
By \eqref{sequenceMl} we have that $\HH^{-1}(M_\ell)= \BB_{-1}$ and $\HH^{0}(M_\ell)= \EE_{\ell}$. Also, we know that $\HH^{-1}(Q)=0$, because $Q$ is a torsion element in the heart. It follows that $\HH^{-1}(P)= \BB_{-1}$ and we have the sequence
$$0 \to \HH^0(P) \to \EE_{\ell} \to \HH^0(Q) \to 0.$$

We recall that $\EE_{\ell}$ is a rank two torsion free sheaf over its support.  Since $\HH^0(P)$ is a subsheaf of $\EE_\ell$, it has the same support. There are three cases: If $\HH^0(P)$ has the same rank of $\EE_\ell$ as a sheaf on its support, then 
\[\chbl(P)=\chbl(M_{\ell}),\]
and $\mu_{\alpha,-1}(Q)=+\infty$, so it is not a destabilizing sequence. The second possibility is that $\HH^0(P)$ has rank 1 and it is torsion free as a sheaf over a line. In this case, we have $\ch_{\BB_0,\leq 2}^{-1}(P)=(-4,3,-1/8)$, whose slope $\mu_{\alpha,-1}$ is less than that of $M_{\ell}$. The third case when $\HH^0(P)=0$ is similar. This proves the stability of $M_\ell$ for $\alpha$ big enough. 
\end{proof}

Now an easy computation using the four conditions listed at the beginning of Section \ref{sec:potential_walls} shows that the only potential wall for $M_\ell$ is given by $\alpha_0=\frac{\sqrt {5}}4$. In the following lemma, we prove that $M_\ell$ remains stable after crossing this wall.

\begin{lem}[The wall for a line]\label{lemma:wallforline}
Let $\alpha_0\geq \frac{\sqrt {5}}4$. If $E$ is a $\sigma_{\alpha_0,-1}$-stable object in $\Psi(\sigma^*\Ku(Y))$ such that $\chbl(E)=(-4,3,\frac{7}8)$, then $E$ is $\sigma_{\alpha,-1}$-stable for any $\alpha>0$.
\end{lem}
\begin{proof}
A direct computation and \cite[Lemma 3.9]{BMS:stabCY3s} imply that the object $E$ can be strictly semistable only with respect to $\sigma_{\frac{\sqrt 5}4,-1}$. If this happens, the Harder-Narasimham filtration of $E$  with respect to $\sigma_{\frac{\sqrt 5}4-\epsilon,-1}$ would be of the form
\[0\subset E_1\subset E\]
with $\chbl(E_1)=  (0,2,1)$ and $\chbl(E/E_1)=(-4,1,-\frac{1}8)$. By Lemma \ref{lem:dlt0obj}, we have that $E/E_1\simeq \BB_0[1]$. In particular, we get $$\Hom(\BB_3,E[3])=(\Hom(E,\BB_0))^*\neq 0,$$ which contradicts to the assumption that $E$ is in $\Psi(\sigma^*\Ku(Y))$. This proves the stability of $E$ as claimed.
\end{proof}

\begin{proof}[Proof of Theorem \ref{thm:Fano_main}]
The first part is a consequence of Lemma \ref{lemma:stablineontop} and Lemma \ref{lemma:wallforline}. The second part follows from the same argument explained in Section \ref{sec:twisted_final} for twisted cubics. We point out that by projecting the universal family, we get an isomorphism from $F_Y$ to $M_\sigma(\lambda_1+\lambda_2)$. Hence the projectivity of $M_\sigma(\lambda_1+\lambda_2)$ follows from that of $F_Y$, without using the result in \cite{liurendapaper}.
\end{proof}

\section{Applications}
\label{section:applications}

In this section we discuss some applications of Theorem \ref{thm:Fano_main} and Theorem \ref{thm:maintwistedcubic}, concerning the categorical version of Torelli Theorem and the derived Torelli Theorem for cubic fourfolds. We also explain the identification of the period point of $M_Y$ with that of $F_Y$.  

\subsection{Torelli Theorem for cubic fourfolds}

In the Appendix of \cite{BLMS:kuzcomponent} the authors gave a different proof of the categorical version of Torelli Theorem for cubic fourfolds introduced in \cite{HuyRen:Torelliforcubicfourfold}, in the case that the algebraic Mukai lattice does not contain $(-2)$-classes, e.g.\ for very general cubic fourfolds. In particular, they deduce the classical version of Torelli Theorem for cubic fourfolds. The key point of their proof is the interpretation of the Fano variety of lines on a very general cubic fourfold as a moduli space of Bridgeland stable objects in the Kuznetsov component.

As a direct consequence of Theorem \ref{thm:Fano_main}, we are able to reprove the categorical formulation of Torelli Theorem for cubic fourfolds without the generality assumption. We recall that the degree shift functor of a cubic fourfold $Y$ is the autoequivalence $(1)$ of $\Ku(Y)$ given by the composition of the tensor product with the line bundle $\OO_Y(1)$ and the projection to $\Ku(Y)$.
\begin{cor}
\label{cor:Torellicat}
Two cubic fourfolds $Y$ and $Y'$ are isomorphic if and only if there is an equivalence between $\Ku(Y)$ and $\Ku(Y')$, whose induced map on the algebraic Mukai lattices commutes with the action of the degree shift functor $(1)$. 
\end{cor}

\subsection{Period point of $M_Y$}

In this section we discuss the relation between the period point of the LLSvS eightfold $M_Y$ associated to a cubic fourfold $Y$ and the period point of $Y$.

As observed in \cite{Debarre-Macri:periodHK}, Example 6.4, the period point of $Y$ is identified with the period point of the Fano variety $F_Y$. More precisely, let ${ }^{2}\MM_6^{(2)}$ be the moduli space of smooth projective hyperk\"ahler fourfolds with a fixed polarization class of degree $6$ and divisibility $2$, deformation equivalent to the Hilbert square of a K3 surface. The Fano variety $F_Y$ with the Pl\"ucker polarization is an element in ${ }^{2}\MM_6^{(2)}$. Let
$$^2p_6^{(2)}: { }^2\MM_6^{(2)} \rightarrow { }^2\PP_6^{(2)}$$
be the period map which is an open embedding by Verbitsky's Torelli Theorem (see \cite{Verbitsky:torelli}). 

Let $H^*(\Ku(Y),\Z)$ be the Mukai lattice of $Y$, which is the orthogonal complement to the lattice spanned by the classes of $\OO_Y$, $\OO_Y(H)$ and $\OO_Y(2H)$ in the topological K-theory of $Y$, with respect to the Euler pairing.
We recall that the embedding of Hodge structures
$$H^2(F_Y,\Z) \rightarrow \langle \lambda_1 \rangle^{\perp} \subset H^*(\Ku(Y),\Z),$$
identifies the polarization class with $\lambda_1+2 \lambda_2$ and $H^2(F_Y,\Z)_{\text{prim}}$ is Hodge isometric to $\langle \lambda_1,\lambda_1 + 2 \lambda_2 \rangle^{\perp}$ (see \cite[Proposition 7]{Addington:rational}).

Using the same notations of \cite{Debarre-Macri:periodHK}, let ${ }^{4}\MM_2^{(2)}$ be the moduli space of smooth projective hyperk\"ahler eightfolds with a fixed polarization class of degree $2$ and divisibility $2$, deformation equivalent to the Hilbert scheme of points of length four on a K3 surface. Let
$$^{4}p_{2}^{(2)}: { }^{4}\MM_2^{(2)} \rightarrow { }^{4}\PP_2^{(2)}$$
be the period map of these eightfolds.

By a direct computation it is possible to show that $M_Y$ carries a natural polarization class of degree $2$ and divisibility $2$. Actually, as observed in \cite{LLMS}, Lemma 3.7, the eightfold $M_Y$ admits a natural antisymplectic involution $\tau$ whose fixed locus contains the cubic fourfold $Y$. Thus, $M_Y$ with the fixed polarization is an element of ${ }^{4}\MM_2^{(2)}$. 

\begin{prop}
\label{prop:period$M_Y$}
Given a cubic fourfold $Y$, We have that
$$^2p_6^{(2)}(F_Y)={ }^4p_2^{(2)}(M_Y)$$
and 
they coincide with the period point of $Y$.
\end{prop}
\begin{proof}
In \cite{liurendapaper} the authors prove that if $M$ is a moduli space of Bridgeland stable objects in $\Ku(Y)$ with Mukai vector $v$ of dimension $2+v^2$, then there is an embedding of Hodge structures
$$H^2(M,\Z) \rightarrow H^*(\Ku(Y),\Z).$$
More precisely, the image of $H^2(M,\Z)$ is identified with the orthogonal complement $v^{\perp}$ of $v$ in the Mukai lattice. Thus, by Theorem \ref{thm:maintwistedcubic}, we have the Hodge isometry
$$H^2(M_Y,\Z) \cong \langle \lambda_1+2 \lambda_2 \rangle^{\perp}.$$
In particular, we can identify the polarization class on $M_Y$ with $\lambda_1$. Then, the primitive degree two lattice $H^2(M_Y,\Z)_{\text{prim}}$ is Hodge isometric to $\langle \lambda_1+2\lambda_2, \lambda_1 \rangle^{\perp}$. It follows that
$$H^2(M_Y,\Z)_{\text{prim}} \cong \langle \lambda_1,\lambda_2 \rangle^{\perp} \cong H^2(F_Y,\Z)_{\text{prim}},$$
which implies the statement.
\end{proof}

As explained in \cite{Debarre:hyper}, Proposition \ref{prop:period$M_Y$} can be used to reprove in a more direct way the result by Laza and Loojenga (see \cite{Laza:image}) about the image of the period map of cubic fourfolds, excluding the divisor of cubic fourfolds containing a plane. This is a work in progress of Bayer and Mongardi.

\subsection{Birational equivalence to Hilbert scheme of points on a K3 surface of $M_Y$}\label{sec:hilb}
Recall that by \cite{AddingtonNicolas2014} the eightfold $M_Y$ is deformation equivalent to a Hilbert scheme of points on a K3 surface. The next application of Theorem \ref{thm:maintwistedcubic} is a characterization of when $M_Y$ is birational to a Hilbert scheme of points on a K3 surface.
\begin{prop}
\label{prop:eightfoldvsHilb}
The hyperk\"ahler eightfold $M_Y$ is birational to a Hilbert scheme $S^{[4]}$ on a K3 surface $S$ if and only if $Y$ is a special cubic fourfold of discriminant $d$ satisfying the condition
\begin{equation}
\label{eq_1}
a^2d=6n^2-6n+2 \quad \text{for } a, n \in \Z.
\end{equation}
\end{prop}
\begin{proof}
By Theorem \ref{thm:maintwistedcubic} the eightfold $M_Y$ is a moduli space of stable objects in the Kuznetsov component of $Y$ with Mukai vector $2 \lambda_1 + \lambda_2$. As a consequence, by \cite{liurendapaper} the degree $2$ cohomology $H^2(M_Y, \Z)$ embeds in the Mukai lattice $\tilde{H}(\Ku(Y),\Z)$ and
$$H^2(M_Y, \Z) \cong \langle 2\lambda_1+\lambda_2 \rangle^{\perp} \subset \tilde{H}(\Ku(Y), \Z).$$

Now, assume that $M_Y$ is birational to $S^{[4]}$. 
By \cite[Proposition 5]{Addington:rational}, there is a element $w \in N(\Ku(Y))$ such that 
\begin{equation}
\label{eq_critAdd}
\chi(w,w)=0 \quad \text{and} \quad \chi(w,2\lambda_1+\lambda_2)=1.
\end{equation}
Here $\chi$ is the Euler pairing and $N(\Ku(Y))$ is the algebraic part of the Mukai lattice of $\Ku(Y)$. Set $n:=\chi(w,\lambda_1)$; by the second equation we have $\chi(w,\lambda_2)=1-2n$.  The lattice $\langle \lambda_1,\lambda_2,w \rangle$ has intersection form given by the matrix
$$
\begin{pmatrix}
-2 & 1 & n \\ 
1 & -2 & 1-2n \\ 
n & 1-2n & 0
\end{pmatrix},$$
whose determinant is $6n^2-6n+2$. Then the saturation of this lattice has discriminant $d$ satisfying $a^2d=6n^2-6n+2$, as we wanted.

Conversely, assume that $Y$ has a discriminant $d$ satisfying \eqref{eq_1}; then $a^2d \equiv 2 \pmod 6$. It follows that $d \equiv 2 \pmod 6$. In particular, by \cite[Lemma 9]{Addington:rational}, there is an element $\tau \in N(\Ku(Y))$ such that $\langle \lambda_1, \lambda_2, \tau \rangle$ has intersection form given by
$$
\begin{pmatrix}
-2 & 1 & 0 \\ 
1 & -2 & 1 \\ 
0 & 1 & 2k
\end{pmatrix} $$
with $d=2+6k$. Moreover, we have $a^2 \equiv 1 \pmod 3$, and thus $a \equiv 1$ or $2 \pmod 3$. 

Assume $a=1+3t$ for $t \in \Z$; we define
$$w:=t \lambda_1 +(2t + n) \lambda_2 + a \tau.$$
It is possible to check that $w$ satisfies conditions \eqref{eq_critAdd}. By \cite[Proposition 5]{Addington:rational}, we conclude that $M_Y$ is birational to a Hilbert scheme over a K3 surface.

Assume $a \equiv 2 \pmod 3 \equiv -1 \pmod 3$; then $a=-1+3t$. We set 
$$w:=-t \lambda_1 +(n-2t) \lambda_2 - a \tau.$$
Again by \cite[Proposition 5]{Addington:rational}, we deduce the statement.
\end{proof}

\begin{rmk}
Note that for $d=14$, the equation \eqref{eq_1} is satisfied with $a=1$ and $n=-1$. This is consistent with \cite{AddingtonNicolas2014} where the authors prove the result in the case of Pfaffian cubic fourfolds.
\end{rmk}

\subsection{Derived Torelli Theorem for cubic fourfolds}
\label{subsec:derTor}
In this section we apply Theorem \ref{thm:maintwistedcubic} to answer a question raised to us by Emanuele Macr\`i. We start with the general motivation:
\begin{question}[Derived Torelli Theorem]
\label{quest:derivedTorelli}
Let $Y$ and $Y'$ be two cubic fourfolds. Is it true that there is a Fourier-Mukai equivalence between $\Ku(Y)$ and $\Ku(Y')$ if and only if there is a Hodge isometry $H^*(\Ku(Y),\Z) \cong H^*(\Ku(Y'),\Z)$?
\end{question}
The question has positive answer for very general cubic fourfolds and cubic fourfolds with associated K3 surface by \cite[Theorem 1.5]{Huybrechts:cubicfourfold} and \cite[Corollary 29.7]{liurendapaper}.
Here we suggest a possible strategy to prove this statement, making use of the description of the eightfold $M_Y$ given by Theorem \ref{mainThm}. For this reason, we need to assume that $Y$ does not contain a plane (in this case the derived Torelli Theorem holds as recalled above).      

Assume that there is a Hodge isometry $\phi: H^*(\Ku(Y),\Z) \cong H^*(\Ku(Y'),\Z)$. Let $v:=2\lambda_1 + \lambda_2$ and we set $v':=\phi(v)$. By \cite{liurendapaper}, the moduli space $M_{\sigma'}(v')$ for $\sigma' \in \Stab(\Ku(Y'))$ is non empty and in particular is a hyperk\"ahler eightfold. By the Birational Torelli Theorem for hyperk\"ahler varieties (see \cite{Verbitsky:torelli}), we have that $M_{\sigma}(v)$ is birational to $M_{\sigma'}(v')$. Thus, by \cite{BM:projectivity}, we can find a stability condition $\sigma''$ such that $M_{\sigma}(v)$ is isomorphic to $M_{\sigma''}(v')$. By the construction in \cite{LLSvS}, the cubic fourfold $Y$ is embedded in $M_{\sigma}(v)$ as a Lagrangian submanifold. Thus, we can see $Y$ inside $M_{\sigma''}(v')$. We denote by $\mathcal{F}$ the restriction of the universal family $\mathcal{E}_C'$ on $M_{\sigma''}(v') \times Y'$ to $Y \times Y'$. We remark that the definition of $\FF$ is up to a twist by a line bundle pulled back from $Y$. The reason will be clear in the proof of Proposition \ref{prop:Y=Y'}.

\begin{question}[Macr\`i]
Does the Fourier-Mukai functor $\Phi_\FF: \Db(Y) \rightarrow \Db(Y')$, defined by $\Phi_{\FF}(-)=q_*(p^*(-) \otimes \FF)$ factorizes through an equivalence $\Ku(Y) \cong \Ku(Y')$ of the the Kuznetsov components?
\end{question}

Although this question remains open, in the simple case where $Y=Y'$ and $\phi \in \text{O}(H^*(\Ku(Y),\Z))$, we can obtain an interesting result, which was originally suggested by Macr\`i. 
Namely we show that the Fourier-Mukai functor $\Phi_{\mathcal{F}}$ commutes with the identity over $\Ku(Y)$.

Denote by $i$ the natural inclusion of $\Ku(Y)$ in $\Db(Y)$, and by $$i \circ i^*:=\R_{\OO_Y(-1)}\L_{\OO_Y}\L_{\OO_Y(1)}: \Db(Y) \to \Ku(Y) \xrightarrow{i} \Db(Y)$$
the projection functor into the Kuznetsov component, changing the notation of the previous sections. Note that we refer to the projection functor in $\Ku(Y)$ with respect to the semiorthogonal decomposition
$$\Db(Y)=\langle \OO_Y(-1), \Ku(Y), \OO_Y, \OO_Y(1) \rangle.$$

\begin{prop}
\label{prop:Y=Y'}
Let $Y$ be a cubic fourfold which does not contain a plane. Then $\Phi_{\mathcal{F}}=i \circ i^*$.
\end{prop}
\begin{proof}
By \cite[Theorem 3.7 and Proposition 3.8]{Kuznetsov2009Hochschild}, the composition $i \circ i^*$ is a Fourier-Mukai functor with kernel given by $\GG:=\text{pr}(\mathcal{O}_{\Delta})$. Here $\OO_{\Delta}$ denotes the structure sheaf of the diagonal in $Y \times Y$ and $\text{pr}: \Db(Y \times Y) \rightarrow \Db(Y) \boxtimes \Ku(Y)$ is the projection functor with respect to the semiorthogonal decomposition
$$\Db(Y \times Y)= \langle \Db(Y) \boxtimes \OO_Y(-1), \Db(Y) \boxtimes \Ku(Y), \Db(Y) \boxtimes \OO_Y , \Db(Y) \boxtimes \OO_Y(1) \rangle.$$

We claim that $\Phi_{\GG}(\OO_y)=\GG_{y}$ is $\sigma$-stable for every $y \in Y$. Indeed, given a point $y$ on the cubic fourfold, there is a non CM twisted cubic curve $C$ on $Y$ which has $y$ as embedded point. In particular, we have the sequence
\begin{equation}
\label{seq_nonCM}
0 \rightarrow \II_{C/Y}(2) \rightarrow \II_{C_0/Y}(2) \rightarrow \OO_y \rightarrow 0,
\end{equation}
where $C_0$ is the plane cubic curve, singular in $y$, defined by $C$. The ideal sheaf of $C_0$ in $Y$ has the following resolution:
\begin{equation}
\label{res_C0}
0 \rightarrow \OO_Y(-1) \rightarrow \OO_Y^{\oplus 3} \rightarrow \OO_Y(1)^{\oplus 3} \rightarrow \II_{C_0/Y}(2) \rightarrow 0.
\end{equation}
We recall that $i^*:=\R_{\OO_Y(-1)}\L_{\OO_Y}\L_{\OO_{Y}(1)}$. We observe that $i^*(\II_{C_0/Y}(2))=0$. Indeed, we split the sequence \eqref{res_C0} in two exact sequences
$$0 \rightarrow \KK \rightarrow \OO_Y(1)^{\oplus 3} \rightarrow \II_{C_0/Y}(2) \rightarrow 0$$
and
$$0 \rightarrow \OO_Y(-1) \rightarrow \OO_Y^{\oplus 3} \rightarrow \KK \rightarrow 0.$$
From the first sequence we get $\L_{\OO_Y(1)}(\II_{C_0/Y}(2)) \cong \L_{\OO_Y(1)}(\KK)[1]$. On the  other hand, $\L_{\OO_Y(1)}$ has not effect on the second sequence, because the objects are in $\langle \OO_Y(1) \rangle^{\perp}$. Applying $\L_{\OO_Y}$, we obtain 
$$\L_{\OO_Y}\L_{\OO_Y(1)}(\KK) \cong \L_{\OO_Y}(\OO_Y(-1))=\OO_Y(-1)[1].$$
It follows that
$$\L_{\OO_Y}\L_{\OO_Y(1)}(\II_{C_0/Y}(2)) \cong \OO_Y(-1)[2].$$
Since $\R_{\OO_Y(-1)}(\OO_Y(-1))=0$, we deduce that $i^*(\II_{C_0/Y}(2))=0$. Thus by the sequence \eqref{seq_nonCM}, we deduce that $i^*(\II_{C/Y}(2)) \cong i^*(\OO_y)[-1]$.

Now, note that $i^*(\II_{C/Y}(2)) \cong i^*(\II_{C/S}(2))$, where $S$ is the cubic surface containing $C$. Indeed, by the resolution
$$0 \rightarrow \OO_Y \rightarrow \OO_Y(1)^{\oplus 2} \rightarrow \II_{S/Y}(2) \rightarrow 0,$$
we see that $\II_{S/Y}(2)$ is in $\langle \OO_Y,\OO_Y(1) \rangle$. Hence, $i^*(\II_{S/Y}(2))=0$. Using the exact sequence
$$0 \rightarrow \II_{S/Y}(2) \rightarrow \II_{C/Y}(2) \rightarrow \II_{C/S}(2) \rightarrow 0,$$
we get 
$$i^*(\II_{C/Y}(2)) \cong i^*(\II_{C/S}(2))= F_C'.$$
By the previous computation, we deduce that $i^*(\OO_y) \cong F_C'[1]$, which is $\sigma$-stable by Theorem \ref{thm:maintwistedcubic}.

It follows that $\GG$ defines an inclusion of $Y$ in the eightfold $M_{\sigma}(v)$ by
$$y \mapsto \Phi_{\GG}(\OO_y).$$
Thus $\GG$ has to be isomorphic to the restriction of the universal family $\EE_C'$ of $M_{\sigma}(v) \times Y$ to $Y \times Y$. Up to a twist of a line bundle on $Y$ pulled back via $p$, we conclude that $\GG \cong \FF$, which gives the statement.
\end{proof}

In the general case, it is expected that the Fourier-Mukai functor $\Phi_{\mathcal{F}}$ factorizes through an equivalence between the Kuznetsov categories. This would give a positive answer to Question \ref{quest:derivedTorelli}.

$$ $$

\bibliography{main}                      
\bibliographystyle{halpha}     
\end{document}